\documentclass[reqno]{amsart}
\usepackage{eurosym}
\usepackage{amsfonts}

\theoremstyle{plain}
\newtheorem{theorem}{Theorem}
\newtheorem{lemma}{Lemma}
\newtheorem{corollary}{Corollary}
\newtheorem{proposition}{Proposition}

\theoremstyle{definition}

\newtheorem{example}{Example}

\theoremstyle{remark}
\newtheorem{remark}{Remark}

\numberwithin{equation}{section}

\begin{document}
\title[Biharmonic Curves in $I\times _{f}M^{n}\left( c\right) $]{Biharmonic
Curves in Warped Product Manifolds $I\times _{f}M^{n}\left( c\right) $}
\author{\c{S}aban G\"{u}ven\c{c}}
\address[\c{S}. G\"{u}ven\c{c}]{Department of Mathematics, Balikesir
University, 10145, Balikesir, T\"{u}rkiye}
\email[\c{S}. G\"{u}ven\c{c}]{sguvenc@balikesir.edu.tr}
\author{Cihan \"{O}zg\"{u}r}
\address[C. \"{O}zg\"{u}r]{Department of Mathematics, \.{I}zmir Democracy
University, 35140, \.{I}zmir, T\"{u}rkiye}
\email[C.~\"{O}zg\"{u}r]{cihan.ozgur@idu.edu.tr}

\begin{abstract}
We explore the geometric properties of biharmonic curves in warped product
manifolds of the form $I\times _{f}M^{n}(c)$, where $I$ is an open interval
and $M^{n}(c)$ is a space of constant curvature. By establishing a main
theorem, we analyze four distinct cases to reveal deeper curvature-related
characteristics of these curves, including situations where they are slant.
Finally, we construct two examples in $I\times _{f}S^{2}(1)$.
\end{abstract}

\maketitle

\section{Introduction}

Let $f :(M,g)\rightarrow (\widetilde{M},\widetilde{g})$ be a smooth map
from a Riemannian manifold $(M,g)$ into an $m$-dimensional Riemannian
manifold $(\widetilde{M},\widetilde{g}).\ f$ is called a \textit{biharmonic
map} if
\begin{equation*}
\tau _{2}(f)=-J^{f}(\tau (f))=-m\left( \Delta ^{f}H-traceR^{\widetilde{M}%
}(df,\tau (f))df\right) =0,
\end{equation*}%
where $J^{f}$ is the Jacobi operator of $f$, $R^{\widetilde{M}}$ denotes the
curvature tensor of $(\widetilde{M},\widetilde{g})$ and $\tau _{2}(f)$ is
called the \textit{bitension field of }$f$ \cite{Jiang-86}.

It is easy to see that if $(\widetilde{M}(c),\widetilde{g})$ is an an $m$%
-dimensional Riemannian space form of constant curvature $c$, then
\begin{equation*}
\tau _{2}(f)=-m\Delta ^{f}H+cm^{2}H.
\end{equation*}%
Thus, $f$ is biharmonic if and only if%
\begin{equation*}
\Delta ^{f}H=cmH,
\end{equation*}%
where $H$ is the mean curvature vector field of $M$. In a different setting,
in \cite{chen-91}, B.-Y. Chen defined a biharmonic submanifold $M\subset
\mathbb{E}^{n}$ of the Euclidean space as its mean curvature vector field $H$
satisfies $\Delta H=0$, where $\Delta $ is the Laplacian (see also \cite%
{chen2014, Chen-2025}). So Jiang's and Chen's definitions on biharmonicity
coincide if the ambient space is the Euclidean space (see \cite{CMC-2002}).

After Jiang's definition, many results have been proven on biharmonic
submanifolds in different ambient spaces. See, for example, \cite{AGR-2013,
BMO-2008, caddeo, Chen-2025, MO-2006, Ou-2010}, and the references therein.

Biharmonic curves have been widely studied in various geometric settings due
to their intrinsic mathematical significance and potential applications in
physics and engineering. See, for example, \cite{chen2014, eellslemaire, eellslemairek},
and the references therein.

Warped product structures have been extensively studied in differential
geometry as well as in physics (see \cite{oneill}). The classification of
biharmonic curves includes cases where they exhibit slant characteristics,
which have been previously studied in different contexts \cite%
{CalinCrasmareanu, CIL-2006, Fetcu, Fetcu-Onic}. Motivated by these
studies, in the present paper, we investigate biharmonic curves in a
specific class of warped product manifolds of the form $I\times _{f}M^{n}(c)$%
, where $I$ is an open interval, $M^{n}(c)$ is an $n$-dimensional space of
constant curvature $c$. We derive the biharmonic equation for curves in $%
I\times _{f}M^{n}(c)$, solve the resulting differential equations, and
analyze the geometric properties of the solutions. Finally, we construct two
explicit examples within $I\times _{f}S^{2}(1)$, supporting our theoretical
results.

This paper is organized as follows: In Section $2$, we provide preliminary
definitions and fundamental results necessary for our analysis. Section $3$
derives the biharmonic equation for curves in $I\times _{f}M^{n}(c)$ and
establishes our main theorem, analyzing four distinct cases. Section $4$
presents two explicit examples supporting our results. In Section $5$, we
present the conclusion of our study along with a discussion on the potential
future directions of research.
\section{Biharmonic Curves in $I\times _{f}M^{n}\left( c\right) $}
Let $M^{n}\left( c\right) =\left( M,g\right) $ denote the $n$-dimensional
simply connected space form of constant curvature $c$ where $g$ is the
Riemannian metric. Let $I\subset
%TCIMACRO{\U{211d} }%
%BeginExpansion
\mathbb{R}
%EndExpansion
$ be an open interval and $f:I\rightarrow
%TCIMACRO{\U{211d} }%
%BeginExpansion
\mathbb{R}
%EndExpansion
$ be a smooth function where $f\left( t\right) \neq 0$ for all $t\in I$. Let
us consider
\begin{equation*}
\widetilde{M}^{n+1}=I\times _{f}M^{n}\left( c\right) =\left( I\times M,%
\widetilde{g}\right)
\end{equation*}%
where
\begin{equation*}
\widetilde{g}=dt^{2}+f^{2}g
\end{equation*}%
is the Riemannian metric of the warped product manifold and $f$ is the
warping function. Notice that we selected the Euclidean metric $dt^{2}$ of
the base interval $I$. Here $t$ denotes the coordinate function of $I$.

\begin{remark}
In a Riemannian manifold context, let us consider an open interval $I$ with
a Riemannian metric $g_{B}$. If we take $g_{B}$ as the Euclidean metric, we
lose no generality, because every Riemannian metric on a one-dimensional
manifold can be transformed into the Euclidean metric via reparametrization.
On a one-dimensional manifold, any Riemannian metric can be written as
\begin{equation*}
g_{B}=\lambda (x)dx^{2},
\end{equation*}%
where $\lambda (x)$ is a positive continuous function of the coordinate
function $x$. However, defining a new parameter
\begin{equation*}
t=\int \sqrt{\lambda (x)}\,dx,
\end{equation*}%
in terms of this new coordinate, the metric takes the standard Euclidean
form $g_{B}=dt^{2}$. Thus, for an open interval $I$, taking the Euclidean
metric does not result in any loss of generality, as any other metric can be
reduced to it through a suitable coordinate change (see \cite{JohnMLee}).
\end{remark}

Now, let us denote the unit vector field tangent to $I$ by $\partial _{t}=%
\frac{\partial }{\partial t}$ and its dual by $\eta $, that is,
\begin{equation*}
\eta \left( X\right) =\widetilde{g}\left( X,\partial _{t}\right) ,\text{ }%
\forall X\in \chi \left( \widetilde{M}^{n+1}\right) .
\end{equation*}%
Any $X\in \chi \left( \widetilde{M}^{n+1}\right) $ can be written as the sum
of two components $X_{1}\in \chi \left( I\right) $ and $X_{2}\in \chi \left(
M\right) ,$ which are called the \textit{horizontal} and the \textit{vertical%
} part of $X$, respectively. As a result, we have%
\begin{equation*}
\widetilde{g}\left( X,Y\right) =X_{1}Y_{1}+f^{2}g\left( X_{2},Y_{2}\right) ,
\end{equation*}%
for $X,Y\in \chi \left( \widetilde{M}^{n+1}\right) ,$ $X=X_{1}+X_{2}$ and $%
Y=Y_{1}+Y_{2}.$ Notice that $\widetilde{g}\left( X_{1},Y_{2}\right) =0$ and $%
\widetilde{g}\left( X_{2},Y_{1}\right) =0.$ The Riemannian curvature tensor
of $\widetilde{M}^{n+1}$ is obtained as%
\begin{eqnarray*}
\widetilde{R}\left( X,Y\right) Z &=&\left[ \left( \frac{f^{\prime }}{f}%
\right) ^{2}-\frac{c}{f^{2}}\right] \left[ \widetilde{g}\left( X,Z\right) Y-%
\widetilde{g}\left( Y,Z\right) X\right] \\
&&+\left[ \frac{f^{\prime \prime }}{f}-\left( \frac{f^{\prime }}{f}\right)
^{2}+\frac{c}{f^{2}}\right] \left[ \widetilde{g}\left( X,Z\right) \eta
\left( Y\right) \partial _{t}-\widetilde{g}\left( Y,Z\right) \eta \left(
X\right) \partial _{t}\right. \\
&&\left. -\eta \left( Y\right) \eta \left( Z\right) X+\eta \left( X\right)
\eta \left( Z\right) Y\right] ,
\end{eqnarray*}%
(see \cite{DDVV5}). We will use the short notation%
\begin{equation*}
f^{\prime }=f^{\prime }\left( t\right) =\frac{d}{dt}f\left( t\right)
\end{equation*}%
and%
\begin{equation*}
\gamma _{i}^{\prime }=\gamma _{i}^{\prime }\left( s\right) =\frac{d}{ds}%
\gamma _{i}\left( s\right) ,\text{ }\left( i=0,1,2,...,n\right)
\end{equation*}%
throughout the paper. Here $t$ denotes the coordinate function of $I$ and $s$
denotes the arc-length parameter of a given curve $\gamma =\gamma \left(
s\right) .$ Also, note that we will write
\begin{equation*}
\theta ^{\prime }=\frac{d}{ds}\theta \left( s\right) ,
\end{equation*}%
where $\theta $ is an angle function along the curve.

Let $\gamma :J\rightarrow \widetilde{M}^{n+1}$ be a unit-speed smooth curve
with arc-length parameter $s\in J$. Then there exists functions $\gamma
_{0}:J\rightarrow I,$ $\gamma _{i}:J\rightarrow
%TCIMACRO{\U{211d} }%
%BeginExpansion
\mathbb{R}
%EndExpansion
,$ $i=1,...,n$, so that $\left( f\circ \gamma \right) \left( s\right)
=f\left( \gamma _{0}\left( s\right) \right) $ is well-defined and $\gamma $
can be written as%
\begin{equation*}
\gamma \left( s\right) =\left( \gamma _{0}\left( s\right) ,\gamma _{1}\left(
s\right) ,...,\gamma _{n}\left( s\right) \right) .
\end{equation*}%
If we denote the coordinate functions of $\widetilde{M}^{n+1}$ by $\left\{
t,x_{1},x_{2},...,x_{n}\right\} ,$ then we can write
\begin{equation*}
t\circ \gamma =\gamma _{0},\text{ }x_{i}\circ \gamma =\gamma _{i},\text{ }%
i=1,...,n.
\end{equation*}%
The tangent vector field of $\gamma $ is
\begin{equation*}
T\left( s\right) =\gamma ^{\prime }\left( s\right) =\left( \gamma
_{0}^{\prime }\left( s\right) ,\gamma _{1}^{\prime }\left( s\right)
,...,\gamma _{n}^{\prime }\left( s\right) \right) ,
\end{equation*}%
which is equivalent to%
\begin{equation*}
T=\gamma _{0}^{\prime }\partial _{t}+\sum_{i=1}^{n}\gamma _{i}^{\prime
}\partial _{x_{i}},
\end{equation*}%
where $\partial _{x_{i}}=\frac{\partial }{\partial x_{i}}.$ Let us denote
the vertical part of $T$ with $W=T-\gamma _{0}^{\prime }\partial _{t}.$
Thus, $\widetilde{g}\left( T,T\right) =1$ gives us
\begin{equation*}
\left( \gamma _{0}^{\prime }\right) ^{2}+f^{2}\left( \gamma _{0}\right)
g\left( W,W\right) =1.
\end{equation*}%
Also notice that $\gamma _{0}^{\prime }=\eta \left( T\right) ,$ since $%
\partial _{t}$ is horizontal and $\partial _{x_{i}}$, $i=1,2,...,n$ are
vertical. We can write
\begin{equation*}
\gamma _{0}^{\prime }=\eta \left( T\right) =\cos \theta ,
\end{equation*}%
where $\theta $ is the angle function between $T$ and $\partial _{t},$ which
is called the \textit{structural angle} \cite{CalinCrasmareanu}. If the
structural angle is constant, we call $\gamma $ a \textit{slant curve} (see
\cite{CalinCrasmareanu}, \cite{CIL-2006} and \cite{Dursun}). If the structural angle vanishes, we
call $\gamma $ a \textit{Legendre curve} \cite{BBlair}, that is, $\eta
\left( T\right) =0$. Note that, if $\gamma $ is an integral curve of $T=\pm
\partial _{t}$, i.e., $\eta \left( T\right) =\cos \theta =\pm 1,$ then it is
a geodesic from the fact that $\widetilde{\nabla }_{T}T=\widetilde{\nabla }%
_{\partial _{t}}\partial _{t}=0.$ In this case, these integral curves of $%
T=\pm \partial _{t}$ are simply the geodesics given by%
\begin{equation*}
\gamma \left( s\right) =\left( \pm s+c_{0},c_{1},...,c_{n}\right) ,
\end{equation*}%
where $c_{i},$ $i=0,1,...,n$ are arbitrary constants. From the definition of
the structural angle, we find%
\begin{equation*}
\gamma _{0}\left( s\right) =\int \cos \theta \left( s\right) ds.
\end{equation*}%
In case $\gamma $ is slant, it is obvious that
\begin{equation}
\gamma _{0}\left( s\right) =s\cos \theta +c_{0}  \label{eq2}
\end{equation}%
for some arbitrary constant $c_{0}$.

If $\gamma :J\rightarrow \widetilde{M}^{n+1}$ is a Frenet curve of
osculating order $r\leq n+1$, one can obtain%
\begin{eqnarray*}
\widetilde{\nabla }_{T}\widetilde{\nabla }_{T}\widetilde{\nabla }_{T}T
&=&-3k_{1}k_{1}^{\prime }T+\left( k_{1}^{\prime \prime
}-k_{1}^{3}-k_{1}k_{2}^{2}\right) E_{2} \\
&&+\left( 2k_{1}^{\prime }k_{2}+k_{1}k_{2}^{\prime }\right)
E_{3}+k_{1}k_{2}k_{3}E_{4},
\end{eqnarray*}%
\begin{eqnarray*}
\widetilde{R}\left( E_{2},T\right) T &=&\left. \left[ \frac{c}{f^{2}}-\left(
\frac{f^{\prime }}{f}\right) ^{2}-\left( \frac{f^{\prime \prime }}{f}-\left(
\frac{f^{\prime }}{f}\right) ^{2}+\frac{c}{f^{2}}\right) \eta \left(
T\right) ^{2}\right] \right\vert _{\gamma }E_{2} \\
&&+\left. \left( \frac{f^{\prime \prime }}{f}-\left( \frac{f^{\prime }}{f}%
\right) ^{2}+\frac{c}{f^{2}}\right) \right\vert _{\gamma }\left[ \eta \left(
T\right) \eta \left( E_{2}\right) T-\eta \left( E_{2}\right) \partial _{t}%
\right] .
\end{eqnarray*}%
As a result, we calculate%
\begin{equation*}
\tau _{2}\left( \gamma \right) =\widetilde{\nabla }_{T}\widetilde{\nabla }%
_{T}\widetilde{\nabla }_{T}T+\widetilde{R}\left( \widetilde{\nabla }%
_{T}T,T\right) T
\end{equation*}%
\begin{equation*}
=\left\{ -3k_{1}k_{1}^{\prime }+k_{1}\left. \left( \frac{f^{\prime \prime }}{%
f}-\left( \frac{f^{\prime }}{f}\right) ^{2}+\frac{c}{f^{2}}\right)
\right\vert _{\gamma }\eta \left( T\right) \eta \left( E_{2}\right) \right\}
T
\end{equation*}%
\begin{equation*}
+\left\{ k_{1}^{\prime \prime }-k_{1}^{3}-k_{1}k_{2}^{2}+k_{1}\left. \left[
\frac{c}{f^{2}}-\left( \frac{f^{\prime }}{f}\right) ^{2}-\left( \frac{%
f^{\prime \prime }}{f}-\left( \frac{f^{\prime }}{f}\right) ^{2}+\frac{c}{%
f^{2}}\right) \eta \left( T\right) ^{2}\right] \right\vert _{\gamma
}\right\} E_{2}
\end{equation*}%
\begin{equation}
+\left( 2k_{1}^{\prime }k_{2}+k_{1}k_{2}^{\prime }\right)
E_{3}+k_{1}k_{2}k_{3}E_{4}-k_{1}\left. \left( \frac{f^{\prime \prime }}{f}%
-\left( \frac{f^{\prime }}{f}\right) ^{2}+\frac{c}{f^{2}}\right) \right\vert
_{\gamma }\eta \left( E_{2}\right) \partial _{t}.  \label{tau2}
\end{equation}%
Now we have our first result:

\begin{proposition}
If $\gamma :J\rightarrow \widetilde{M}^{n+1}$ is a unit-speed biharmonic
curve, then $k_{1}$ is a constant.
\end{proposition}

\begin{proof}
If $\gamma $ is biharmonic, then $\tau _{2}\left( \gamma \right) =0.$ If we
use equation (\ref{tau2}) and apply $T,$ we find%
\begin{equation*}
-3k_{1}k_{1}^{\prime }=0.
\end{equation*}%
Then, $k_{1}=0$ or $k_{1}^{\prime }=0.$ In both cases, $k_{1}$ is a constant.
\end{proof}

\begin{remark}
If $\gamma $ is a geodesic, (i.e. $k_{1}=0$), then $\tau _{2}\left( \gamma
\right) $ vanishes, so $\gamma $ is obviously biharmonic. Thus, we need to
focus on non-geodesic biharmonic curves.
\end{remark}

\begin{lemma}
Let $\gamma :J\rightarrow \widetilde{M}^{n+1}$ be a non-geodesic unit-speed
Frenet curve of osculating order $r$. Then $\gamma $ is biharmonic if and
only if $k_{1}$ is a constant and%
\begin{equation*}
\left. \left( \frac{f^{\prime \prime }}{f}-\left( \frac{f^{\prime }}{f}%
\right) ^{2}+\frac{c}{f^{2}}\right) \right\vert _{\gamma }\eta \left(
T\right) \eta \left( E_{2}\right) T
\end{equation*}%
\begin{equation*}
\left\{ -k_{1}^{2}-k_{2}^{2}+\left. \left[ \frac{c}{f^{2}}-\left( \frac{%
f^{\prime }}{f}\right) ^{2}-\left( \frac{f^{\prime \prime }}{f}-\left( \frac{%
f^{\prime }}{f}\right) ^{2}+\frac{c}{f^{2}}\right) \eta \left( T\right) ^{2}%
\right] \right\vert _{\gamma }\right\} E_{2}
\end{equation*}%
\begin{equation*}
+k_{2}^{\prime }E_{3}+k_{2}k_{3}E_{4}-\left. \left( \frac{f^{\prime \prime }%
}{f}-\left( \frac{f^{\prime }}{f}\right) ^{2}+\frac{c}{f^{2}}\right)
\right\vert _{\gamma }\eta \left( E_{2}\right) \partial _{t}=0.
\end{equation*}
\end{lemma}

Now, we can state our main theorem:

\begin{theorem}
\label{maintheorem}Let $\gamma :J\rightarrow \widetilde{M}^{n+1}$ be a
unit-speed Frenet curve of osculating order $r$ and let $m=min\left\{
r,4\right\} $. Then $\gamma $ is biharmonic if and only

$i)$ $\left. \left[ \frac{f^{\prime \prime }}{f}-\left( \frac{f^{\prime }}{f}%
\right) ^{2}+\frac{c}{f^{2}}\right] \right\vert _{\gamma }=0$ or $\eta
\left( E_{2}\right) =0$ or $\partial _{t}\in sp\left\{
T,E_{2},E_{3},E_{4}\right\} ;$ and

$ii)$ the first $m$ of the following equations are satisfied along the curve
$\gamma $ (replacing $k_{m}=0$):%
\begin{equation*}
k_{1}=constant,
\end{equation*}%
\begin{equation*}
k_{1}^{2}+k_{2}^{2}=\left. \frac{c}{f^{2}}\right\vert _{\gamma }-\left(
\left. \frac{f^{\prime }}{f}\right\vert _{\gamma }\right) ^{2}-\left. \left(
\frac{f^{\prime \prime }}{f}-\left( \frac{f^{\prime }}{f}\right) ^{2}+\frac{c%
}{f^{2}}\right) \right\vert _{\gamma }\left[ \eta \left( T\right) ^{2}+\eta
\left( E_{2}\right) ^{2}\right] ,
\end{equation*}%
\begin{equation*}
k_{2}^{\prime }=\left. \left( \frac{f^{\prime \prime }}{f}-\left( \frac{%
f^{\prime }}{f}\right) ^{2}+\frac{c}{f^{2}}\right) \right\vert _{\gamma
}\eta \left( E_{2}\right) \eta \left( E_{3}\right) ,
\end{equation*}%
\begin{equation*}
k_{2}k_{3}=\left. \left( \frac{f^{\prime \prime }}{f}-\left( \frac{f^{\prime
}}{f}\right) ^{2}+\frac{c}{f^{2}}\right) \right\vert _{\gamma }\eta \left(
E_{2}\right) \eta \left( E_{4}\right) .
\end{equation*}
\end{theorem}

\begin{proof}
The proof is straightforward using equation (\ref{tau2}). $\gamma $ is
biharmonic if and only $\tau _{2}\left( \gamma \right) =0$. Since $\left\{
T,E_{2},E_{3},E_{4}\right\} $ is orthonormal, the critical term with $%
\partial _{t}$ determines the conditions of i). The equations of ii) are
obtained applying $T,E_{2},E_{3}$ and $E_{4},$ that is, $\widetilde{g}\left(
\tau _{2}\left( \gamma \right) ,E_{i}\right) =0,$ $i=1,2,...,m$.
\end{proof}

We will now analyze this theorem under four different cases.

\textbf{Case I.} $\left. \left[ \frac{f^{\prime \prime }}{f}-\left( \frac{%
f^{\prime }}{f}\right) ^{2}+\frac{c}{f^{2}}\right] \right\vert _{\gamma }=0.$

In this case, the equations of Theorem \ref{maintheorem} become%
\begin{equation*}
k_{1}=constant,
\end{equation*}%
\begin{equation*}
k_{1}^{2}+k_{2}^{2}=-\left. \frac{f^{\prime \prime }}{f}\right\vert _{\gamma
},
\end{equation*}%
\begin{equation*}
k_{2}^{\prime }=0,
\end{equation*}%
\begin{equation*}
k_{2}k_{3}=0.
\end{equation*}

Thus, we observe that $\gamma $ is a geodesic, a circle or a helix. The
global solution of the ODE%
\begin{equation*}
f^{\prime \prime }+\left( k_{1}^{2}+k_{2}^{2}\right) f=0
\end{equation*}%
is $f\left( t\right) =c_{1}\sin \left( \sqrt{k_{1}^{2}+k_{2}^{2}}t\right)
+c_{2}\cos \left( \sqrt{k_{1}^{2}+k_{2}^{2}}t\right) .$ If we write this
solution in the ODE
\begin{equation*}
ff^{\prime \prime }-\left( f^{\prime }\right) ^{2}+c=0
\end{equation*}%
for a common solution, we find%
\begin{equation*}
\left( c_{1}^{2}+c_{2}^{2}\right) \left( k_{1}^{2}+k_{2}^{2}\right) =c,
\end{equation*}%
which gives the following result:

i) No global real solution for $\widetilde{M}=I\times _{f}\mathbb{E}^{n}$ $%
\left( c=0\right) $ or $\widetilde{M}=I\times _{f}\mathbb{H}^{n}\left(
-1\right) $ $\left( c=-1\right) .$

ii) $f(t)=\frac{1}{\sqrt{k_{1}^{2}+k_{2}^{2}}}\sin \left( \sqrt{%
k_{1}^{2}+k_{2}^{2}}t+c_{0}\right) $ for $\widetilde{M}=I\times _{f}\mathbb{S%
}^{n}\left( 1\right) $ $\left( c=1\right) .$ Here $c_{0}$ is an arbitrary
constant. We used a phase shift as $\tan c_{0}=\frac{c_{2}}{c_{1}}$ for
simplification.

\begin{corollary}
\label{cor1}Let $\gamma :J\rightarrow I\times _{f}\mathbb{S}^{n}\left(
1\right) $ be a unit-speed helix with curvatures $k_{1}$ and $k_{2}$, or a
circle with curvature $k_{1}$. If
\begin{equation*}
f(t)=\frac{1}{\sqrt{k_{1}^{2}+k_{2}^{2}}}\sin \left( \sqrt{%
k_{1}^{2}+k_{2}^{2}}t+c_{0}\right)
\end{equation*}%
for some constant $c_{0}$ (replacing $k_{2}=0$ in case $\gamma $ is a
circle), then $\gamma $ is biharmonic.
\end{corollary}

\begin{remark}
The converse statement is not always valid because even though the ODEs
might not be globally satisfied, they might be satisfied along the curve $%
\gamma .$ In fact, if $\gamma $ is a Legendre curve, then $\gamma
_{0}=constant$. Composing the ODEs with a constant function yields a
constant value, making it more likely to hold locally. Thus Case I i)
statement does not mean that there is no non-geodesic biharmonic curve in $%
\widetilde{M}=I\times _{f}\mathbb{E}^{n}$ or $\widetilde{M}=I\times _{f}%
\mathbb{H}^{n}\left( -1\right) $ within this case.
\end{remark}

\textbf{Case II.} $\eta \left( E_{2}\right) =0.$

Let us revisit the equation%
\begin{equation*}
T=\gamma _{0}^{\prime }\partial _{t}+W,
\end{equation*}%
where we denote the vertical part with%
\begin{equation*}
W=\sum_{i=1}^{n}\gamma _{i}^{\prime }\partial _{x_{i}}=T-\gamma _{0}^{\prime
}\partial _{t}.
\end{equation*}%
From the Levi-Civita connection of the warped product manifold, we have%
\begin{equation*}
\widetilde{\nabla }_{T}\partial _{t}=\left. \frac{f^{\prime }}{f}\right\vert
_{\gamma }W.
\end{equation*}%
We also have%
\begin{eqnarray*}
\widetilde{g}\left( T,W\right) &=&\widetilde{g}\left( T,T-\gamma
_{0}^{\prime }\partial _{t}\right) \\
&=&1-\left( \gamma _{0}^{\prime }\right) ^{2}=1-\cos ^{2}\theta \\
&=&\sin ^{2}\theta .
\end{eqnarray*}%
Differentiating $\eta \left( T\right) =\cos \theta $ , we find%
\begin{eqnarray*}
-\theta ^{\prime }\sin \theta &=&\widetilde{\nabla }_{T}\widetilde{g}\left(
T,\partial _{t}\right) \\
&=&\widetilde{g}\left( \widetilde{\nabla }_{T}T,\partial _{t}\right) +%
\widetilde{g}\left( T,\widetilde{\nabla }_{T}\partial _{t}\right) \\
&=&k_{1}\eta \left( E_{2}\right) +\left. \frac{f^{\prime }}{f}\right\vert
_{\gamma }\widetilde{g}\left( T,W\right) \\
&=&k_{1}\eta \left( E_{2}\right) +\left. \frac{f^{\prime }}{f}\right\vert
_{\gamma }\sin ^{2}\theta .
\end{eqnarray*}%
Thus, we can give the following Lemma:

\begin{lemma}
For a non-Legendre unit-speed Frenet curve of osculating order $r\geq 2$ in $%
I\times _{f}M^{n}\left( c\right) ,$ $E_{2}$ and $\partial _{t}$ are
orthogonal if and only if the warping function satisfies%
\begin{equation}
\left. f\right\vert _{\gamma }=c_{1}\csc \theta  \label{lemmaa}
\end{equation}%
for some arbitrary constant $c_{1}$. Moreover, if the curve is slant, then $%
f $ is a constant along the curve.
\end{lemma}

\begin{proof}
From the equation%
\begin{equation}
-\theta ^{\prime }\sin \theta =k_{1}\eta \left( E_{2}\right) +\left. \frac{%
f^{\prime }}{f}\right\vert _{\gamma }\sin ^{2}\theta ,  \label{lemm}
\end{equation}%
$E_{2}$ and $\partial _{t}$ are orthogonal if and only if%
\begin{equation*}
\theta ^{\prime }\sin \theta +\left. \frac{f^{\prime }}{f}\right\vert
_{\gamma }\sin ^{2}\theta =0.
\end{equation*}%
If $\sin \theta =0$, then $T=\pm \partial _{t}$ and $\gamma $ becomes a
geodesic. This contradicts $r\geq 2.$ Thus we have
\begin{equation*}
\frac{\theta ^{\prime }}{\sin \theta }+\left. \frac{f^{\prime }}{f}%
\right\vert _{\gamma }=0.
\end{equation*}%
More clearly, this equation can be written as%
\begin{equation*}
\frac{\theta ^{\prime }}{\sin \theta }+\frac{f^{\prime }\left( \gamma
_{0}\right) }{f\left( \gamma _{0}\right) }=0,
\end{equation*}%
where $\gamma _{0}\left( s\right) =t\circ \gamma \left( s\right) $ is the
first component of $\gamma $ and $s$ is the arc-length parameter. If we
multiply the equation with $\gamma _{0}^{\prime }=\eta \left( T\right) =\cos
\theta \neq 0,$ we obtain%
\begin{equation*}
\frac{\theta ^{\prime }\cos \theta }{\sin \theta }+\frac{f^{\prime }\left(
\gamma _{0}\right) }{f\left( \gamma _{0}\right) }\gamma _{0}^{\prime }=0.
\end{equation*}%
Now we can integrate with respect to $s$ and find (\ref{lemmaa}). If $\theta
$ is a constant (i.e. $\gamma $ is slant), we deduce that $f$ is a constant
along the curve.
\end{proof}

The following proposition holds for slant curves, regardless of $E_{2}$ and $%
\partial _{t}$ are orthogonal.

\begin{proposition}
\label{propslant}If $\gamma :J\rightarrow I\times _{f}M^{n}\left( c\right) $
is a non-Legendre slant curve, then%
\begin{equation*}
\log c_{1}=\int k_{1}\eta \left( E_{2}\right) ds+\frac{\sin ^{2}\theta }{%
\cos \theta }\log \left\vert f\left( s\cos \theta +c_{0}\right) \right\vert .
\end{equation*}
\end{proposition}

\begin{proof}
The proof follows from equation (\ref{lemm}), considering $\gamma _{0}\left(
s\right) =s\cos \theta +c_{0}.$ Since $\theta $ is a constant for a slant
curve, equation (\ref{lemm}) becomes%
\begin{equation*}
0=k_{1}\eta \left( E_{2}\right) +\frac{f^{\prime }\left( \gamma _{0}\right)
}{f\left( \gamma _{0}\right) }\sin ^{2}\theta .
\end{equation*}%
From $\gamma _{0}^{\prime }=\cos \theta \neq 0,$ we can write
\begin{equation*}
0=k_{1}\eta \left( E_{2}\right) +\frac{f^{\prime }\left( \gamma _{0}\right)
\gamma _{0}^{\prime }}{f\left( \gamma _{0}\right) }\frac{\sin ^{2}\theta }{%
\cos \theta }.
\end{equation*}%
Then we integrate the last equation along the curve $\gamma $ and obtain the
result.
\end{proof}

After careful investigation, let us give the biharmonicity corollary of Case
II:

\begin{corollary}
Let $\gamma :J\rightarrow I\times _{f}M^{n}\left( c\right) $ be a
non-Legendre unit-speed helix with curvatures $k_{1}$ and $k_{2}$, or a
non-Legendre circle with curvature $k_{1}$. If $\left. f\right\vert _{\gamma
}=c_{1}\csc \theta $ for some arbitrary constant $c_{1},$ then $\gamma $ is
biharmonic if and only if its curvatures satisfy%
\begin{equation}
k_{1}^{2}+k_{2}^{2}=\left. \frac{c}{f^{2}}\right\vert _{\gamma }-\left(
\left. \frac{f^{\prime }}{f}\right\vert _{\gamma }\right) ^{2}-\left. \left(
\frac{f^{\prime \prime }}{f}-\left( \frac{f^{\prime }}{f}\right) ^{2}+\frac{c%
}{f^{2}}\right) \right\vert _{\gamma }\cos ^{2}\theta .  \label{eq1}
\end{equation}
\end{corollary}

Moreover, the following corollary holds for non-Legendre slant curves:

\begin{corollary}
Let $\gamma :J\rightarrow I\times _{f}M^{n}\left( c\right) $ be a unit-speed
non-Legendre slant helix with curvatures $k_{1}$ and $k_{2}$, or a
non-Legendre slant circle with curvature $k_{1}$. If $\left. f\right\vert
_{\gamma }=c_{1}\csc \theta $=constant, then $\gamma $ is biharmonic if and
only if $c>0$ and
\begin{equation*}
k_{1}^{2}+k_{2}^{2}=\frac{c}{c_{1}^{2}}\sin ^{4}\theta .
\end{equation*}
\end{corollary}

\begin{proof}
Using the fact that the structural angle is a non-zero constant (i.e. $\cos
\theta \neq 0$) for a non-Legendre slant curve and equation (\ref{eq2}),
after differentiating $\left. f\right\vert _{\gamma }=c_{1}\csc \theta $%
=constant, we find%
\begin{equation*}
\frac{d}{ds}\left( \left. f\right\vert _{\gamma }\right) =\left. \frac{df}{dt%
}\right\vert _{\gamma }.\gamma _{0}^{\prime }\left( s\right) =\left.
f^{\prime }\right\vert _{\gamma }\cos \theta =0.
\end{equation*}%
This results $\left. f^{\prime }\right\vert _{\gamma }=0$. In a similar way,
we obtain $\left. f^{\prime \prime }\right\vert _{\gamma }=0.$ Thus, the
right-hand side of equation (\ref{eq1}) becomes $\frac{c}{c_{1}^{2}}\sin
^{4}\theta ,$ which completes the proof.
\end{proof}

\textbf{Case III.} $E_{2}$ is parallel to $\partial _{t},$ i.e. $\eta \left(
E_{2}\right) =\pm 1.$

In this case, we can write $E_{2}=\varepsilon \partial _{t}$ where $%
\varepsilon $ denotes $\pm 1.$ This is the same as $\partial
_{t}=\varepsilon E_{2}.$ For this case, it is directly seen that%
\begin{equation}
\eta \left( T\right) =\widetilde{g}\left( T,\partial _{t}\right) =\widetilde{%
g}\left( T,\varepsilon E_{2}\right) =0  \label{eq3}
\end{equation}%
and%
\begin{equation*}
\eta \left( E_{2}\right) =\widetilde{g}\left( E_{2},\partial _{t}\right) =%
\widetilde{g}\left( E_{2},\varepsilon E_{2}\right) =\varepsilon .
\end{equation*}%
Equation (\ref{eq3}) means that $\gamma $ is a Legendre curve, that is $\cos
\theta =0.$ Considering $T$ is vertical, we have%
\begin{equation*}
\widetilde{\nabla }_{T}\partial _{t}=\left. \frac{f^{\prime }}{f}\right\vert
_{\gamma }T.
\end{equation*}%
So, if we differentiate $E_{2}=\varepsilon \partial _{t}$ along $\gamma ,$
we get%
\begin{equation*}
-k_{1}T+k_{2}E_{3}=\varepsilon \widetilde{\nabla }_{T}\partial
_{t}=\varepsilon \left. \frac{f^{\prime }}{f}\right\vert _{\gamma }T,
\end{equation*}%
which results
\begin{equation}
k_{1}=-\varepsilon \left. \frac{f^{\prime }}{f}\right\vert _{\gamma },\text{
}k_{2}=0.  \label{eq4}
\end{equation}%
Now we can state our next corollary for Case III:

\begin{corollary}
\label{cor4}Let $\gamma :J\rightarrow I\times _{f}M^{n}\left( c\right) $ be
a unit-speed Legendre circle with $E_{2}\parallel \partial _{t}.$ Then $%
\gamma $ is biharmonic if and only if its curvature satisfies%
\begin{equation}
k_{1}=\sqrt{\left. \frac{-f^{\prime \prime }}{f}\right\vert _{\gamma }}%
=-\varepsilon \left. \frac{f^{\prime }}{f}\right\vert _{\gamma }>0.
\label{eqq}
\end{equation}
\end{corollary}

\begin{proof}
Since $\eta \left( T\right) =\cos \theta =0,$ $\eta \left( E_{2}\right)
=\varepsilon =\pm 1,$ from the second equation of Theorem \ref{maintheorem}
ii), we can write%
\begin{equation}
k_{1}^{2}=\left. \frac{c}{f^{2}}\right\vert _{\gamma }-\left( \left. \frac{%
f^{\prime }}{f}\right\vert _{\gamma }\right) ^{2}-\left. \left( \frac{%
f^{\prime \prime }}{f}-\left( \frac{f^{\prime }}{f}\right) ^{2}+\frac{c}{%
f^{2}}\right) \right\vert _{\gamma }=-\left. \frac{f^{\prime \prime }}{f}%
\right\vert _{\gamma }.  \label{eq5}
\end{equation}%
Also, $\cos \theta =0$ gives us $\gamma _{0}\left( s\right) =c_{0}=constant,$
where $c_{0}\in I.$ Then equations (\ref{eq4}) and (\ref{eq5}) complete the
proof.
\end{proof}

\begin{remark}
A natural equation to consider in this context is the ODE
\begin{equation*}
-f\left( t\right) f^{\prime \prime }\left( t\right) =\left( f^{\prime
}\left( t\right) \right) ^{2},
\end{equation*}%
whose solution is%
\begin{equation}
f\left( t\right) =c_{2}\sqrt{2t+c_{1}},  \label{eq6}
\end{equation}%
where $c_{1}$ and $c_{2}$ are arbitrary constants. Thus, if $f$ is of the
form (\ref{eq6}), it would satisfy
\begin{equation}
\sqrt{\left. \frac{-f^{\prime \prime }}{f}\right\vert _{\gamma }}%
=-\varepsilon \left. \frac{f^{\prime }}{f}\right\vert _{\gamma }>0,
\label{eq7}
\end{equation}%
for any $\gamma _{0}\left( s\right) =c_{0}.$ But, since (\ref{eq7}) is not
necessarily global, it suffices for any $f$ (even not in the form of (\ref%
{eq6})) to satisfy equation (\ref{eq7}) for the first component of the
curve, i.e. $\gamma _{0}\left( s\right) =c_{0}\in I,$ rather than
everywhere. Then a Legendre circle whose curvature $k_{1}$ satisfies (\ref%
{eqq}), becomes biharmonic in this warped product manifold.
\end{remark}

\textbf{Case IV.} $\left. \left[ \frac{f^{\prime \prime }}{f}-\left( \frac{%
f^{\prime }}{f}\right) ^{2}+\frac{c}{f^{2}}\right] \right\vert _{\gamma
}\neq 0$ and $\eta \left( E_{2}\right) \neq 0,\pm 1.$

In this case, from Theorem \ref{maintheorem}, we have $\partial _{t}\in
sp\left\{ T,E_{2},E_{3},E_{4}\right\} .$ So, we can write%
\begin{equation}
\partial _{t}=\cos \theta T\pm \sin \theta \cos w_{1}\cos w_{2}E_{2}\pm \sin
\theta \cos w_{1}\sin w_{2}E_{3}\pm \sin \theta \sin w_{1}E_{4}  \label{eq8}
\end{equation}%
for three angle functions $\theta ,$ $w_{1}$ and $w_{2}$ along the curve.
Here $\theta $ is the structural angle. $w_{1}$ denotes the angle between $%
\left( \partial _{t}-\cos \theta T\right) $ and its projection onto $%
sp\left\{ E_{2},E_{3}\right\} .$ $w_{2}$ denotes the angle between $E_{2}$
and the projection of $\left( \partial _{t}-\cos \theta T\right) $ onto $%
sp\left\{ E_{2},E_{3}\right\} $. Then, we obtain%
\begin{equation*}
\eta \left( E_{2}\right) =\pm \sin \theta \cos w_{1}\cos w_{2},
\end{equation*}%
\begin{equation*}
\eta \left( E_{3}\right) =\pm \sin \theta \cos w_{1}\sin w_{2},
\end{equation*}%
\begin{equation*}
\eta \left( E_{4}\right) =\pm \sin \theta \sin w_{1}.
\end{equation*}%
Thus, we can state our final corollary as follows:

\begin{corollary}
Let $\gamma :J\rightarrow \widetilde{M}^{n+1}$ be a unit-speed Frenet curve
of osculating order $r\geq 4,$ $\eta \left( E_{2}\right) \neq 0,\pm 1$ and $%
\left. \left[ ff^{\prime \prime }-\left( f^{\prime }\right) ^{2}+c\right]
\right\vert _{\gamma }\neq 0.$ Then $\gamma $ is biharmonic if and only if $%
k_{1}$ is a positive constant and%
\begin{equation*}
k_{1}^{2}+k_{2}^{2}=\left. \frac{c}{f^{2}}\right\vert _{\gamma }-\left(
\left. \frac{f^{\prime }}{f}\right\vert _{\gamma }\right) ^{2}-\left. \left(
\frac{f^{\prime \prime }}{f}-\left( \frac{f^{\prime }}{f}\right) ^{2}+\frac{c%
}{f^{2}}\right) \right\vert _{\gamma }\left[ \cos ^{2}\theta +\sin
^{2}\theta \cos ^{2}w_{1}\cos ^{2}w_{2}\right] ,
\end{equation*}%
\begin{equation*}
k_{2}^{\prime }=\frac{1}{2}\left. \left( \frac{f^{\prime \prime }}{f}-\left(
\frac{f^{\prime }}{f}\right) ^{2}+\frac{c}{f^{2}}\right) \right\vert
_{\gamma }\sin ^{2}\theta \cos ^{2}w_{1}\sin 2w_{2},
\end{equation*}%
\begin{equation*}
k_{2}k_{3}=\frac{1}{2}\left. \left( \frac{f^{\prime \prime }}{f}-\left(
\frac{f^{\prime }}{f}\right) ^{2}+\frac{c}{f^{2}}\right) \right\vert
_{\gamma }\sin ^{2}\theta \sin 2w_{1}\cos w_{2},
\end{equation*}%
where $\partial _{t}\in sp\left\{ T,E_{2},E_{3},E_{4}\right\} $ is given by
equation $(\ref{eq8})$.
\end{corollary}

\begin{proof}
The proof is straightforward using Theorem \ref{maintheorem} with equation (%
\ref{eq8}).
\end{proof}

\section{Construction of Examples in $I\times _{f}S^{2}\left( 1\right) $}

Let us consider $\widetilde{M}^{n+1}=I\times _{f}S^{2}\left( 1\right) ,$
where $f\left( t\right) =\sin \left( t\right) $ and $S^{2}\left( 1\right)
\subset E^{3}$ is the unit-sphere in Euclidean space. We will construct a
biharmonic curve in $\widetilde{M}^{n+1}$ using Corollary \ref{cor1}.
Firstly, notice that $f$ is of the form%
\begin{equation*}
f(t)=\frac{1}{\sqrt{k_{1}^{2}+k_{2}^{2}}}\sin \left( \sqrt{%
k_{1}^{2}+k_{2}^{2}}t+c_{0}\right) ,
\end{equation*}%
for $c_{0}=0$ and $k_{1}^{2}+k_{2}^{2}=1.$ For simplicity, let us look for a
biharmonic circle, that is, $\kappa _{1}=1$ and $\kappa _{2}=0.$ We will use
spherical coordinate functions $\left\{ x_{1},x_{2}\right\} $ of $%
S^{2}\left( 1\right) .$ Here $x_{1}\in \left( 0,\pi \right) $ and $x_{2}\in
\left( 0,2\pi \right) .$ Recall that the induced metric tensor field of $%
S^{2}\left( 1\right) $ can be written as
\begin{equation*}
g=dx_{1}^{2}+\cos ^{2}x_{1}dx_{2}^{2},
\end{equation*}%
and the corresponding Levi-Civita Connection is%
\begin{equation*}
\nabla _{\partial _{x_{1}}}\partial _{x_{1}}=0,
\end{equation*}%
\begin{equation*}
\nabla _{\partial _{x_{1}}}\partial _{x_{2}}=\nabla _{\partial
_{x_{2}}}\partial _{x_{1}}=-\tan x_{1}\partial _{x_{2}},
\end{equation*}%
\begin{equation*}
\nabla _{\partial _{x_{2}}}\partial _{x_{2}}=\frac{1}{2}\sin 2x_{1}\partial
_{x_{1}}.
\end{equation*}%
As a result, we can write the metric tensor field of $I\times
_{f}S^{2}\left( 1\right) $ as%
\begin{eqnarray*}
\widetilde{g} &=&dt^{2}+f^{2}g \\
&=&dt^{2}+\sin ^{2}\left( t\right) \left( dx_{1}^{2}+\cos
^{2}x_{1}dx_{2}^{2}\right) .
\end{eqnarray*}%
So, the corresponding Levi-Civita Connection is calculated as%
\begin{equation*}
\widetilde{\nabla }_{\partial _{t}}\partial _{t}=0,
\end{equation*}%
\begin{equation*}
\widetilde{\nabla }_{\partial _{t}}\partial _{x_{1}}=\widetilde{\nabla }%
_{\partial _{x_{1}}}\partial _{t}=\cot \left( t\right) \partial _{x_{1}},
\end{equation*}%
\begin{equation*}
\widetilde{\nabla }_{\partial _{t}}\partial _{x_{2}}=\widetilde{\nabla }%
_{\partial _{x_{2}}}\partial _{t}=\cot \left( t\right) \partial _{x_{2}},
\end{equation*}%
\begin{equation*}
\widetilde{\nabla }_{\partial _{x_{1}}}\partial _{x_{1}}=-\frac{1}{2}\sin
\left( 2t\right) \partial _{t},
\end{equation*}%
\begin{equation*}
\widetilde{\nabla }_{\partial _{x_{1}}}\partial _{x_{2}}=\widetilde{\nabla }%
_{\partial _{x_{2}}}\partial _{x_{1}}=-\tan x_{1}\partial _{x_{2}},
\end{equation*}%
\begin{equation*}
\widetilde{\nabla }_{\partial _{x_{2}}}\partial _{x_{2}}=-\frac{1}{2}\sin
\left( 2t\right) \cos ^{2}x_{1}\partial _{t}+\frac{1}{2}\sin \left(
2x_{1}\right) \partial _{x_{1}}.
\end{equation*}%
Now, let us define the curve $\gamma :J\rightarrow I\times _{f}S^{2}\left(
1\right) ,$
\begin{equation*}
s\mapsto \gamma \left( s\right) =\left( \frac{\pi }{4},\sqrt{2}s,\pi \right)
.
\end{equation*}%
Its tangent vector field is%
\begin{equation*}
T=\gamma ^{\prime }=\left( 0,\sqrt{2},0\right) =\sqrt{2}\partial _{x_{1}}.
\end{equation*}%
Notice that
\begin{equation*}
\widetilde{g}\left( \gamma ^{\prime },\gamma ^{\prime }\right) =\widetilde{g}%
\left( \sqrt{2}\partial _{x_{1}},\sqrt{2}\partial _{x_{1}}\right) =2\left.
\sin ^{2}\left( t\right) \right\vert _{\gamma }=2\sin ^{2}\left( \frac{\pi }{%
4}\right) =1,
\end{equation*}%
namely, $\gamma $ is unit-speed and $s$ is the arc-length parameter. Then we
have%
\begin{equation*}
\widetilde{\nabla }_{T}T=\widetilde{\nabla }_{\sqrt{2}\partial _{x_{1}}}%
\sqrt{2}\partial _{x_{1}}=2\widetilde{\nabla }_{\partial _{x_{1}}}\partial
_{x_{1}}=-\left. \sin \left( 2t\right) \right\vert _{\gamma }\partial
_{t}=-\partial _{t}.
\end{equation*}%
As a result, we find%
\begin{equation*}
\kappa _{1}=\sqrt{\widetilde{g}\left( \widetilde{\nabla }_{T}T,\widetilde{%
\nabla }_{T}T\right) }=1
\end{equation*}%
and%
\begin{equation*}
E_{2}=-\partial _{t}.
\end{equation*}%
Thus, it is easy to see that%
\begin{equation}
\widetilde{\nabla }_{T}E_{2}=\widetilde{\nabla }_{\sqrt{2}\partial
_{x_{1}}}\left( -\partial _{t}\right) =-\sqrt{2}\widetilde{\nabla }%
_{\partial _{x_{1}}}\partial _{t}=-\sqrt{2}\left. \cot \left( t\right)
\right\vert _{\gamma }\partial _{x_{1}}=-\sqrt{2}\partial _{x_{1}}.
\label{*}
\end{equation}%
From the Frenet formulas, we know that%
\begin{equation}
\widetilde{\nabla }_{T}E_{2}=-\kappa _{1}T+\kappa _{2}E_{3}=-\sqrt{2}%
\partial _{x_{1}}+\kappa _{2}E_{3}.  \label{**}
\end{equation}%
Finally, equations (\ref{*}) and (\ref{**}) give us $\kappa _{2}=0$. We have
shown that $\gamma $ is a circle in $I\times _{f}S^{2}\left( 1\right) $
satisfying Corollary \ref{cor1}. So, it is biharmonic. Moreover, notice that
$E_{2}\parallel \partial _{t}$. So, we obtained an example for Case III,
too. Let us show that Corollary \ref{cor4} is verified. Since $f\left(
t\right) =\sin t,$ we can write%
\begin{equation*}
\left. \frac{f^{\prime \prime }}{f}\right\vert _{\gamma }=\left. \left(
\frac{-\sin t}{\sin t}\right) \right\vert _{\gamma }=-1
\end{equation*}%
and%
\begin{equation*}
\left. \frac{f^{\prime }}{f}\right\vert _{\gamma }=\left. \left( \frac{\cos t%
}{\sin t}\right) \right\vert _{\gamma }=\cot \left( \frac{\pi }{4}\right) =1.
\end{equation*}%
As a result, we obtain%
\begin{equation*}
k_{1}=\sqrt{\left. \frac{-f^{\prime \prime }}{f}\right\vert _{\gamma }}%
=-\varepsilon \left. \frac{f^{\prime }}{f}\right\vert _{\gamma }=1,
\end{equation*}%
which is true for $\gamma ,$ a circle with $E_{2}\parallel \partial _{t}$.
Thus, Corollary \ref{cor4} is also verified and $\gamma $ is biharmonic. To
summarize, we write:

\begin{example}
$\gamma :J\rightarrow I\times _{f}S^{2}\left( 1\right) ,$
\begin{equation*}
s\mapsto \gamma \left( s\right) =\left( \frac{\pi }{4},\sqrt{2}s,\pi \right)
\end{equation*}%
is a biharmonic cirle in the warped product manifold $I\times
_{f}S^{2}\left( 1\right) ,$ where $f(t)=\sin t$ is the warping function.
\end{example}

Finally, we will give another non-trivial example in $I\times
_{f}S^{2}\left( 1\right) $ with $f(t)=\sin t$.

\begin{example}
Let $a,b\in
%TCIMACRO{\U{211d} }%
%BeginExpansion
\mathbb{R}
%EndExpansion
$ satisfying $a^{2}+b^{2}=1$. Then $\gamma :J\rightarrow I\times
_{f}S^{2}\left( 1\right) ,$
\begin{equation*}
s\mapsto \gamma \left( s\right) =\left( \arccos \left( \cos \left( as\right)
\cos \left( bs\right) \right) ,\arcsin \left( \frac{\cos \left( as\right)
\sin \left( bs\right) }{\sqrt{1-\cos ^{2}\left( as\right) \cos ^{2}\left(
bs\right) }}\right) ,bs\right)
\end{equation*}%
becomes a unit-speed helix of order $3$ $\left( k_{3}=0\right) $ with
constant curvatures $k_{1}=\left\vert \sin \left( 2u\right) \right\vert $
and $k_{2}=\left\vert \cos \left( 2u\right) \right\vert $, where $u\in
%TCIMACRO{\U{211d} }%
%BeginExpansion
\mathbb{R}
%EndExpansion
$ is a constant such that $a=\cos u$ and $b=\sin u$. In this case, we have $%
k_{1}^{2}+k_{2}^{2}=1$. In view of Corollary \ref{cor1}, the equation%
\begin{equation*}
f(t)=\frac{1}{\sqrt{k_{1}^{2}+k_{2}^{2}}}\sin \left( \sqrt{%
k_{1}^{2}+k_{2}^{2}}t+c_{0}\right)
\end{equation*}%
is satisfied along the curve $\gamma $ with $c_{0}=0.$ So, $\gamma $ is
biharmonic.
\end{example}

\section{Conclusion and Discussion}

In this paper, we study biharmonic curves in warped product manifolds of the
form $I\times _{f}M^{n}(c)$. By deriving the biharmonic equation and
analyzing the solutions, we identified key geometric properties such as
slant characteristics.
This study suggests potential directions for future research, inspired by
previous works, particularly in extending the results to more general warped
product manifolds or higher-dimensional spaces.


\begin{thebibliography}{99}
\bibitem{AGR-2013} Al\'{\i}as LJ., Garc\'{\i}a-Mart\'{\i}nez S, Carolina R.
M. Biharmonic hypersurfaces in complete Riemannian manifolds. Pacific J.
Math. 2013; 263 (1): 1--12.

\bibitem{BBlair} Baikoussis C, Blair DE. On Legendre curves in contact
3-manifolds. Geom. Dedicata 1994; 49 (2): 135-142.

\bibitem{BMO-2008} Balmu\c{s} A., Montaldo S., Oniciuc C. Classification
results for biharmonic submanifolds in spheres. Israel J. Math. 2008; 168:
201--220.

\bibitem{caddeo} Caddeo R, Montaldo S, Oniciuc C. Biharmonic submanifolds of
$\mathbb{S}^{3}$. Internat. J. Math. 2001; 12 (8): 867-876.

\bibitem{CMC-2002} Caddeo R, Montaldo S, Oniciuc C. Biharmonic submanifolds
in spheres. Israel J. Math. 2002; 130: 109--123.

\bibitem{CalinCrasmareanu} Calin C, Crasmareanu M. Slant curves and
particles in three-dimensional warped products and their Lancret invariants.
Bull. Aust. Math. Soc. 2013; 88 (1): 128-142.

\bibitem{chen-91} Chen B.-Y. Some open problems and conjectures on
submanifolds of finite type. Soochow J. Math. 1991; 17 (2): 169-188.

\bibitem{chen2014} Chen B.-Y. Some open problems and conjectures on
submanifolds of finite type: recent development. Tamkang J. Math. 2014; 45
(1): 87-108. 

\bibitem{Chen-2025} Chen B.-Y. Recent developments in Chen's biharmonic
conjecture and some related topics. Mathematics. 2025; 13:1417.

\bibitem{CIL-2006} Cho JT, Inoguchi JI, Lee JE. On slant curves in Sasakian
3-manifolds. Bull. Austral. Math. Soc. 2006; 74 (3): 359-367.

\bibitem{Dursun} Dursun U. Slant curves in the Lorentzian warped product
manifold $-I\times \mathbb{E}^{2}$. J. Geom., 2022; 113 (1): Paper No. 24,
18 pp.

\bibitem{eellslemaire} Eells J, Lemaire L. A report on harmonic maps. Bull.
London Math. Soc. 1978; 10 (1): 1-68. 

\bibitem{eellslemairek} Eells J, Lemaire L. Selected Topics in Harmonic
Maps. CBMS Regional Conference Series in Mathematics, Vol. 50. Providence,
RI: American Mathematical Society, 1983. 

\bibitem{Fetcu} Fetcu D. Biharmonic Legendre curves in Sasakian space forms.
J. Korean Math. Soc. 2008; 45 (2): 393-404.

\bibitem{Fetcu-Onic} Fetcu D, Oniciuc C. Explicit formulas for biharmonic
submanifolds in Sasakian space forms. Pacific J. Math. 2009; 240 (1):
85-107.

\bibitem{Guv} \"{O}zg\"{u}r C., G\"{u}ven\c{c} \c{S}. On biharmonic Legendre
curves in $S$-space forms. Turkish J. Math. 2014; 38(3): 454--461.

\bibitem{Jiang-86} Jiang, GY. 2-harmonic maps and their first and second
variational formulas. Chinese Ann. Math. Ser. A . 1986; 7(4): 389--402.

\bibitem{JohnMLee} Lee JM. Riemannian Manifolds: An Introduction to
Curvature. New York: Springer-Verlag, 1997.

\bibitem{DDVV5} Lawn MA, Ortega M. A fundamental theorem for hypersurfaces
in semi-Riemannian warped products. J. Geom. Phys. 2015; 90 : 55-70.

\bibitem{MO-2006} Montaldo S.; Oniciuc C. A short survey on biharmonic maps
between Riemannian manifolds. Rev. Un. Mat. Argentina. 2006; 47 (2): 1--22.

\bibitem{oneill} O'Neill B. Semi-Riemannian Geometry with Applications to
Relativity. New York-London: Academic Press, 1983.

\bibitem{Ou-2010} Ou, Y-L. Biharmonic hypersurfaces in Riemannian manifolds.
Pacific J. Math. 248 (2010), no. 1, 217--232.

\end{thebibliography}
\end{document}